\documentclass[a4paper]{amsart}
\usepackage[T1]{fontenc}
 \usepackage[francais, english]{babel}
 \usepackage{amssymb,url,xspace}
 \usepackage{graphicx}
  \usepackage[all]{xy}
\theoremstyle{plain}

\usepackage[all]{xy}

\newtheorem{thm}{Theorem}

\newtheorem{prop}[thm]{Proposition}

\theoremstyle{definition}

\newtheorem{example}[thm]{Example}

\def\Con{\mathcal{C}\mathit{on}}

\def\Fuchs{\mathcal{F}\mathit{uchs}}

\begin{document}

\author[F. Loray]{Frank LORAY}
\address{ IRMAR - UMR 6625 du CNRS, Campus de Beaulieu, Universit\'e de Rennes 1, 35042 Rennes Cedex, France.}
\email{frank.loray@univ-rennes1.fr}

\title{Introduction to moduli spaces of connections: some explicit constructions}

\begin{abstract}We give some concrete examples of moduli spaces of connections.
Precisely, we explain how to explicitely construct the moduli spaces 
of rank $2$ fuchsian systems and logarithmic connections on the Riemann sphere 
with $4$ poles. The former ones are affine cubic surfaces; the latter ones are analytically
isomorphic to affine surfaces but not algebraically: they do not carry non constant 
regular functions. We end with some remark on arbitrary number of poles.
\end{abstract}


\maketitle

\section{Introduction: variation on an example of Serre}

Given a complex smooth projective curve $X$, a rank $1$ (flat) connection,
or equivalently a rank $1$ local system on $X$, is an element of the cohomology 
group $H^1(X,\mathbb C^*)$. This is the easiest example of a moduli space of connections. 
From the exact sequence of sheaves
$$0\to\mathbb  C^*\to \mathcal O^*\to \Omega^1\to 0$$
(where $\mathcal O^*\to \Omega^1$ is given by $f\mapsto\frac{df}{f}$),
we can derive the associate long exact sequence of cohomology groups and extract the following
\begin{equation}\label{Eq:ExactSeqRk1}
0\to H^0(X,\Omega^1) \to H^1(X,\mathbb C^*) \to \mathrm{Jac}(X)\to 0
\end{equation}
making $H^1(X,\mathbb C^*)$ into a principal $\mathbb C^g$-bundle over the Jacobian of the curve.

Even simpler is the case when $X$ is an elliptic curve: we get a $\mathbb C$-bundle over the curve
$X\simeq \mathrm{Jac}(X)$ itself. One may compactify it by adding a section at infinity in order to get
a $\mathbb P^1$-bundle $P\to X$: it is one of the two undecomposable $\mathbb P^1$-bundles found 
by Atiyah in \cite{Atiyah}. The section at infinity $\sigma:X\to P$ is the unique section whose image $\Sigma:=\sigma(X)$
has zero self-intersection in the total space $S:=\mathrm{Tot}(P)$ of the bundle. 
The group $H^1(X,\mathbb C^*)$ identifies, as a variety, with the complement $S\setminus\Sigma$.

Viewing $X$ as a Riemann surface, we get the Riemann-Hilbert correspondence 
$$RH\ :\ H^1(X,\mathbb C^*)\longrightarrow \mathrm{Hom}(\pi_1(X),\mathbb C^*)$$
which associates to a local system its monodromy representation. Once we choose 
a system of generators for the fundamental group, we get an identification between
the space of representations and the group $\mathbb C^*\times\mathbb C^*$.
We thus get a complex analytic group isomorphism
\begin{equation}\label{Eq:RHaffineRk1}
H^1(X,\mathbb C^*)\longrightarrow \mathbb C^*\times\mathbb C^*.
\end{equation}
However, the two underlying varieties are not isomorphic from the algebraic point of view:
 while $\mathbb C^*\times\mathbb C^*$
is an affine variety, there is no non constant regular functions on $S$.
Indeed, such a function would extend as a rational function on $S$ with polar divisor $n\Sigma$;
one easily check that the zero divisor of such function should define a complete curve not intersecting $\Sigma$.
But there is no complete curve in $S$ since it is analytically isomorphic to 
an affine variety. This famous example is due to Serre.

There is a natural non degenerate $2$-form $\omega$ on $S$ constructed from the 
exact sequence (\ref{Eq:ExactSeqRk1}) by choosing non zero $1$-forms on the base $X$ and on the fiber;
it is well-defined up to a scalar constant. Isomorphism (\ref{Eq:RHaffineRk1}) sends it to 
the $2$-form $\frac{du}{u}\wedge\frac{dv}{v}$ (up to a scalar). The $2$-form $\omega$ extends on $S$
as a rational $2$-form with divisor $2\Sigma$.

Let us briefly mention two other facts. The maximal compact subgroup given by unitary connections/representation
is an analytic real torus providing a real section of $P\to X$, and turns to be real algebraic on representation-side,
namely $\mathbb S^1\times\mathbb S^1\subset\mathbb C^*\times\mathbb C^*$. The fibration $P\to X$
is sent to the foliation defined by $ u\partial_u+\tau v\partial_v$ where $(1:\tau)$ is commensurable
to periods of $X$.  

One can modify Serre's example as follows. The elliptic involution $\iota:X\to X$ lifts-up as a biregular involution
$\varphi:S\to S$.
This follows from the unicity of the undecomposable bundle $P$, or equivalently from the natural
action of the involution $\iota$ on local systems. The quotient surface $\underline{S}$ 
fibers over $X/\iota\simeq\mathbb P^1$.
The involution $\varphi$ has $2$ fixed points over each of the $4$ Weierstrass points: one of them lie
on the section at infinity $\Sigma$ (which is invariant) and the other one is corresponding to a $2$-torsion 
point of $H^1(X,\mathbb C^*)$. We thus get $8$ conic singular points on the quotient that we have to blow-up.
The resulting surface $\hat{\underline{S}}\to \underline{S}$ is equipped with the quotient $2$-form $\underline{\omega}$
having polar divisor $2\hat{\underline{\Sigma}}+E_1+E_2+E_3+E_4$ where $E_i$ are exceptional divisors
of conic points along $\underline{\Sigma}$ (the image of $\Sigma$); $\underline{\omega}$
is non degenerate on the open set $M:=\hat{\underline{S}}\setminus(\hat{\underline{\Sigma}}\cup E_1\cup E_2\cup E_3\cup E_4)$.
This open set is no more a group, but share all other nice properties with Serre example.
For instance, the divisor $2\hat{\underline{\Sigma}}+E_1+E_2+E_3+E_4$ at infinity is (from the numerical point of view) 
like a degenerate elliptic fiber of type $I_0^*$ in Kodaira list,
there is no non constant regular function on $M$ and it is analytically conjugated to an affine surface, namely the Cayley cubic, quotient of $\mathbb C^*\times\mathbb C^*$
by the involution $(u,v)\mapsto(\frac{1}{u},\frac{1}{v})$.
This second example is actually a particular case of moduli space of connections arising as space of initial conditions of one of the Painlev\'e VI
equations, namely that one found by Picard.

These nice properties are common to general moduli spaces of connections (except structure of group, particular for the rank one) and this is a good motivation
to study them. On the other hand, they naturally appear as spaces of initial conditions for isomonodromy
equations, but we skip this point of view from our dicussion. 

\section{Moduli space of fuchsian systems}\label{sec:FuchsianSyst}

Consider a $\mathfrak{sl}_2$-fuchsian system on the Riemann sphere with $4$ poles
\begin{equation}\label{Eq:FuchSyst}
\frac{dY}{dx}=\frac{A_1}{x-t_1}+\frac{A_2}{x-t_2}+\frac{A_3}{x-t_3}+\frac{A_4}{x-t_4}
\end{equation}
with constant matrices $A_i$, $i=1,2,3,4$, satisfying 
\begin{equation}\label{Eq:FuchSyst2}
A_i=\begin{pmatrix}a_i&b_i\\ c_i&-a_i\end{pmatrix}\in\mathfrak{sl}_2(\mathbb C)\ \ \ \text{and}\ \ \ A_1+A_2+A_3+A_4=0
\end{equation}
(here and after we assume all $4$ poles $t_i$'s in the affine part for simplicity).
The group $\mathrm{SL}_2(\mathbb C)$ acts on the $Y$-variable and thus on residues $A_i$'s by simultaneous conjugacy.
Indeed, change of variable $Y=MY'$ induces change of residues ${A}_i'=M^{-1}A_iM$.
The spectrum of each matrix $A_i$ is preserved by this action. Let us fix the spectral data
\begin{equation}\label{Eq:FuchSystSpectral}
-\det(A_i)=a_i^2+b_ic_i=\theta_i^2
\end{equation}
for some $\boldsymbol{\theta}=(\theta_1,\theta_2,\theta_3,\theta_4)$ (eigenvalues are $\pm\theta_i$).
We want to describe the quotient space $\Fuchs^{\boldsymbol{\theta}}(X,D)$ of those systems by this action:
\begin{equation}\label{Eq:DefFuchSystQuotient}
\left\{(A_1,A_2,A_3,A_4)\in(\mathfrak{sl}_2(\mathbb C))^4\ ;\ \sum_iA_i=0,\ \det(A_i)=-\theta_i^2\right\}/_{\mathrm{SL}_2(\mathbb C)}.
\end{equation}
A straightforward computation shows that $\Fuchs^{\boldsymbol{\theta}}(X,D)$ is expected
to be a surface (depending on $4$ parameters $\theta_i$'s). However, it is non Hausdorff (as a topological space) 
in general. 
For instance, the orbit of the triangular system 
\begin{equation}\label{Eq:FuchSystTriang}
A_i=\begin{pmatrix}\theta_i&b_i\\ 0&-\theta_i\end{pmatrix}
\end{equation}
under diagonal conjugacy is not closed; its closure contains the diagonal system (with all $b_i$'s vanishing)
so that they define infinitesimally closed points in the quotient. Note however that, in this case, we have $\theta_1+\theta_2+\theta_3+\theta_4=0$. 

Introduce the following functions on $\Fuchs^{\boldsymbol{\theta}}(X,D)$:
\begin{equation}\label{Eq:FuchSystInvDef}
\begin{matrix}X_1:=\det(A_2+A_3),\ X_2:=\det(A_1+A_3),\ X_3:=\det(A_1+A_2)\\
\text{and}\ \ \ Y=\mathrm{tr}\left(A_1[A_2,A_4]\right)\end{matrix}
\end{equation}
where $[A,B]=AB-BA$ is the Lie bracket. These functions are clearly invariant under $\mathrm{SL}_2(\mathbb C)$-action.

\begin{prop}\label{Prop:FuchsInv}
Assume $\theta_4\not=0$. Then, the map
\begin{equation}\label{Eq:FuchSystInvMap}
(X_1,X_2,X_3,Y)\ :\ \Fuchs^{\boldsymbol{\theta}}(X,D)\longrightarrow\mathbb C^4
\end{equation}
sends the moduli space onto the affine cubic surface $S^{\boldsymbol{\theta}}$ defined by 
\begin{equation}\label{Eq:FuchSystCubic}
\begin{matrix}
X_1+X_2+X_3 = \theta_1^2+\theta_2^2+\theta_3^2+\theta_4^2\ \ \ \text{and}\\
\frac{Y^2}{4}+X_1X_2X_3+(\theta_1^2-\theta_3^2)(\theta_2^2-\theta4^2)X_1+(\theta_2^2-\theta_3^2)(\theta_1^2-\theta4^2)X_2\\
=(\theta_1^2+\theta_2^2-\theta_3^2-\theta4^2)(\theta_1^2\theta_2^2-\theta_3^2\theta4^2)
\end{matrix}
\end{equation}
The map (\ref{Eq:FuchSystInvMap}) above is one-to-one over the smooth part of $S^{\boldsymbol{\theta}}$.
Singularities arise when
\begin{itemize}
\item $\theta_i=0$ and $A_i=0$ for $i=1,2,3$;
\item $\pm\theta_1\pm\theta_2\pm\theta_3\pm\theta_4=0$ and all $A_i$'s are simultaneously triangular up to conjugacy.
\end{itemize}
In particular, apart from special values of $\boldsymbol{\theta}$ listed just before, the quotient $\Fuchs^{\boldsymbol{\theta}}(X,D)$
is Hausdorff and is isomorphic to the smooth cubic surface $S^{\boldsymbol{\theta}}$.
\end{prop}

\begin{proof}Since $\theta_4\not=0$, we can assume $A_4=\begin{pmatrix}\theta_4&0\\0&-\theta_4\end{pmatrix}$.
This normalization is well-defined up to diagonal conjugacy so that the monomials $a_i$ and $b_ic_j$ are invariant.
Using $A_1+A_2+A_3+A_4=0$, we can now express $A_3$ in function of $A_1$ and $A_2$:
$$\left\{\begin{matrix}a_3&=&-a_1-a_2-\theta_4\\ b_3&=&-b_1-b_2\\ c_3&=&-c_1-c_2\end{matrix}\right.$$
Spectral data (\ref{Eq:FuchSystSpectral}) gives the following conditions
\begin{equation}\label{Eq:FuchSystInvRelat}
\begin{matrix}
a_1^2+b_1c_1=\theta_1^2,\ \ \ a_2^2+b_2c_2=\theta_2^2\ \ \ \text{and}\\
2a_1a_2+2\theta_4(a_1+a_2)+b_1c_2+b_2c_1+\theta_1^2+\theta_2^2-\theta_3^2+\theta4^2=0.
\end{matrix}
\end{equation}
On the other hand, we get 
\begin{equation}\label{Eq:FuchSystInvExplicitxi}
\begin{matrix}X_1&=&2\theta_4a_1+\theta_4^2+\theta_1^2\\ X_2&=&2\theta_4a_2+\theta_4^2+\theta_2^2\\ X_3&=&2\theta_4a_3+\theta_4^2+\theta_3^2\end{matrix}
\ \ \ \text{and}\ \ \ Y=4\theta_4(b_1c_2-b_2c_1)
\end{equation}
Once we know $X_1,X_2,X_3,Y$, we get the invariants $a_1$, $a_2$ and $b_1c_2-b_2c_1$.
From relations (\ref{Eq:FuchSystInvRelat}), we promptly deduce $b_1c_1$, $b_2c_2$ and $b_1c_2+b_2c_1$,
and therefore $b_1c_2$ and $b_2c_1$.
The cubic equation (\ref{Eq:FuchSystCubic}) just says that these invariants satisfy the obvious relation
$(b_1c_1)(b_2c_2)=(b_1c_2)(b_2c_1)$. We can thus recover $(b_1,b_2,c_1,c_2)$ uniquely up to the diagonal action,
except when either $b_1=b_2=0$, or $c_1=c_2=0$; in these latter cases, there are several possible solutions
$(A_1,A_2,A_3)$ up to diagonal action, but all of them are triangular.
\end{proof}

\section{Moduli space of connections}

Fix $D=t_1+\cdots+t_n$ a reduced divisor (of poles) on the Riemann sphere $X:=\mathbb P^1$ with $n\ge4$.
Fix eigenvalues $\{\theta_i^+,\theta_i^-\}\subset \mathbb C$ for each $i=1,\ldots,n$ with integral sum
$\sum_{i=1}^n\theta_i^++\theta_i^-=-d\in\mathbb Z$. For simplicity, assume generic condition 
$\theta_1^{\pm}+\cdots+\theta_n^{\pm}\not\in\mathbb Z$ for any choice of signs $\pm$
to avoid reducible connections. Denote by $\boldsymbol{\theta}=(\theta_1^{\pm},\ldots,\theta_n^{\pm})$ the spectral data.

Consider the triples $(E,\nabla,\boldsymbol{l})$ called ``parabolic connections'' where 
\begin{itemize}
\item $E$ is a rank $2$ vector bundle of degree $d$ over $\mathbb P^1$, 
\item $\nabla$ is a logarithmic connection $\nabla:E\to E\otimes\Omega^1(D)$
with polar divisor $D$, having residual eigenvalues $\theta_i^+$ and $\theta_i^-$ over the pole $t_i$,
\item $\boldsymbol{l}=(l_1,\ldots,l_n)$ is a parabolic structure on $E$
such that $l_i\subset E\vert_{t_i}$ belongs to the eigenspace generated by $\theta_i^+$.
\end{itemize}
When $\theta_i^+\not=\theta_i^-$, we note that the parabolic $l_i$ is {\bf the} eigendirection of $\theta_i^+$;
the parabolic structure  is therefore relevant only in case of equality $\theta_i^+=\theta_i^-$. 
Note also that Fuchs relation says that 
$\sum_{i=1}^n\theta_i^++\theta_i^-+\deg(E)=0$, which explains the constraints given by $d$ in the above definitions.

We say that two parabolic connections $(E,\nabla,\boldsymbol{l})$ and $(E',\nabla',\boldsymbol{l}')$
are equivalent when there is a bundle isomorphism $\phi:E\to E'$ conjugating connections and parabolic structures.
We denote by $\Con^{\boldsymbol{\theta}}(X,D)$ the moduli space of parabolic connections for this equivalence. 
This can be viewed
as a stack (see \cite{ArinkinLysenko}), but it can actually be constructed by GIT method (see \cite{IIS}); under our 
generic assumptions on $\boldsymbol{\theta}$, we simply get:

\begin{thm}[Inaba-Iwasaki-Saito]\label{thm:IIS}
The moduli space  $\Con^{\boldsymbol{\theta}}(X,D)$ is a smooth irreducible quasi-projective variety
of dimension $2(n-3)$ equipped with a regular symplectic $2$-form $\omega$.
\end{thm}

The construction needs a choice of weights to impose stability condition; all parabolic connections
are stable under our assumptions and the resulting quotient does not actually depends on this choice.
However, if we allow non generic eigenvalues, then some reducible connections become unstable, and the 
quotient of semi-stable points depends on the choice of weights. Note also that we could have avoid 
parabolic structure in the discussion by assuming also $\theta_i^+\not=\theta_i^-$, but we will use 
parabolics later; this is why we already introduce this notion.

\subsection{Connection matrix}\label{subsec:ConMatrix}
Let us now be more explicit for readers that are not familiar to connections.
Following Birkhoff, any rank vector bundle $E$ on $\mathbb P^1$ splits as a direct sum of line bundles.
In particular, for the rank $2$ case, we have 
$$E=\mathcal O(d_1)\oplus\mathcal O(d_2)\ \ \ \text{with}\ \ \ d_1\le d_2\ \ \ \text{and}\ \ \ d_1+d_2=d.$$
Choose sections $e_i$ of $\mathcal O(d_i)$ whose divisor (zero or pole) is supported by $\infty\in\mathbb P^1$
so that, over the affine part $\mathbb C=\mathbb P^1\setminus\{\infty\}$, the vector bundle is trivial: $E\vert_{\mathbb C}=\mathbb C e_1\oplus\mathbb C e_2$.
We can describe the connection $\nabla$ as follows:
$$\nabla\vert_{\mathbb C}:Y\mapsto dY+\Omega\cdot Y$$
where $Y$ is a section of $E\vert_{\mathbb C}$ and $\Omega$ is a $2\times2$-matrix of meromorphic $1$-forms:
\begin{equation}\label{Eq:MatrixCon}
\Omega=\sum_{i=1}^n\frac{A_i}{x-t_i}dx+B(x)dx
\end{equation}
with $B$ holomorphic (and $x$ the affine variable of $\mathbb C=\mathbb P^1\setminus\{\infty\}$). 
Let $e_i'$ be sections of $\mathcal O(d_i)$ defined by $e_i=x^{d_i}e_i'$. At infinity, we must express
the connection in the basis $(e_1',e_2')$ and, setting
$$Y=MY'\ \ \ \text{with}\ \ \ M=\begin{pmatrix}x^{d_1}&0\\0&x^{d_2}\end{pmatrix}$$
we find
$$\nabla\vert_{\mathbb P^1\setminus\{0\}}:Y'\mapsto dY'+\Omega'\cdot Y'\ \ \ \text{with}\ \ \ \Omega'=M^{-1}\Omega M+M^{-1}dM$$
which must be holomorphic at $\infty$. If we now write $\Omega=A(x)\frac{dx}{\prod_{i=1}^n(x-t_i)}$, this holomorphy
conditions says that
$$A(x)=-\begin{pmatrix}{d_1}&0\\0&{d_2}\end{pmatrix}x^{n-1}+\begin{pmatrix}a(x)&b(x)\\c(x)&d(x)\end{pmatrix}$$
$$\text{with}\ \ \ 
\left\{\begin{matrix}b\ \text{polynomial of degree}\ \le n-2-(d_2-d_1) \\ 
a,d\ \text{polynomials of degree}\ \le n-2\hfill\\ c\ \text{polynomial of degree}\ \le n-2+(d_2-d_1)\end{matrix}\right.$$
Fuchsian systems are just logarithmic connections on the trivial bundle $d_1=d_2=0$.

\begin{example}[Case $n=4$]\label{ex:NonTrivBunCon}
When $E=\mathcal O\oplus\mathcal O(1)$, then holomorphy at infinity gives the following constraints
$$A_1+A_2+A_3+A_4=\begin{pmatrix}0&0\\ \star&-1\end{pmatrix}\ \ \ \text{and}\ \ \ B(x)\equiv 0.$$
When $E=\mathcal O(-1)\oplus\mathcal O(1)$, then we get constraints
$$A_1+A_2+A_3+A_4=\begin{pmatrix}1&0\\ \star&-1\end{pmatrix},\ 
B(x)\equiv\begin{pmatrix}0&0\\ \star&0\end{pmatrix}\
\text{(constant matrices)}$$
$$\text{and}\ \ \ \sum_i\frac{A_i}{x-t_i}=\begin{pmatrix}\star&\frac{b}{\prod_i(x-t_i)}\\ \star&\star\end{pmatrix}.$$
Finally, when $d_2-d_1>2$, the $[1,2]$-coefficient of all $A_i$'s and $B(x)$ must be zero so that we are in the reducible
case. In other words, under our generic assumption on $\boldsymbol{\theta}$, we have $d_2-d_1=0,1,2$.
For degree $d=0$ or $1$, we are led to the following possibilities
\begin{itemize}
\item $d=0$ and $(d_1,d_2)=(0,0)$ or $(-1,1)$,
\item $d=1$ and $(d_1,d_2)=(0,1)$.
\end{itemize}
\end{example}

We now impose spectral data by 
$$\mathrm{tr}(A_i)=\theta_1^++\theta_i^-\ \ \ \text{and}\ \ \ \det(A_i)=\theta_1^+\cdot\theta_i^-.$$
The parabolic structure is therefore given by 
$$l_i=\mathbb C\cdot\begin{pmatrix}b_i\\ \theta_i^+-a_i\end{pmatrix}\ \ \ \text{or}\ \ \ 
\mathbb C\cdot\begin{pmatrix} \theta_i^+-d_i\\ c_i\end{pmatrix}\ \ \ \text{where}\ \ \ 
A_i=\begin{pmatrix}a_i&b_i\\ c_i&d_i\end{pmatrix}$$
except when $\theta_i^+=\theta_i^-$ and $A_i$ is a scalar matrix, in which case $l_i$
may be arbitrarily choosen: it is an extra data in this case.

\subsection{Bundle automorphisms}
So far, we have described the space of parabolic connections up to Birkhoff normalization 
of the bundle. Then, we would like to quotient by the action of bundle automorphisms on connections.
They can be described as follows.
\begin{itemize}
\item When $d_1=d_2$, then $E$ is just the twist of the trivial bundle by $\mathcal O(d_1)$
and automorphisms are the same: $\mathrm{GL}_2(\mathbb C)$ acts by conjugacy on the connection matrix (\ref{Eq:MatrixCon}),
i.e. on matrices $A_i$'s ($B$ is scalar and the action is trivial on it).
\item When $d_1<d_2$, then the automorphism group is given in the trivialization chart $E\vert_{\mathbb C}=\mathbb C e_1\oplus\mathbb C e_2$ by
$$\left\{M=\begin{pmatrix}\lambda_1&0\\ f(x)&\lambda_2\end{pmatrix}\ \text{with}\ f(x)\ \text{polynomial of degree}\ \le d_2-d_1\right\},$$
and it is acting on connection matrix (\ref{Eq:MatrixCon}) by 
$$\Omega\mapsto \Omega'=M^{-1}\Omega M+M^{-1}dM.$$
\end{itemize}
The action on parabolics is obvious.

\begin{example}\label{ex:ModuliConF2}
Assume $n=4$ and $(d_1,d_2)=(-1,1)$; for simplicity, assume also  $(\theta_i^+,\theta_i^-)=(\theta_i,-\theta_i)$.
The matrix connection writes $\Omega=A(x)\frac{dx}{\prod_i(x-t_i)}$ where
$$A(x)=\begin{pmatrix}x^3+a(x)&b\\c(x)&-x^3-a(x)\end{pmatrix} \ \ \ \text{with}\ \ \ 
\left\{\begin{matrix}b\in\mathbb C\hfill\\ a\ \text{polynomial of degree}\ \le2\\ 
c\ \text{polynomial of degree}\ \le4\end{matrix}\right.$$
Note that $b\not=0$ otherwise the connection would be reducible.
Conjugating by a diagonal matrix $M$, we can assume $b=1$. The action of unipotent isomorphisms is described by
$$M=\begin{pmatrix}1&0\\f(x)&1\end{pmatrix}\ :\ \ \ \begin{pmatrix}x^3+a&1\\c&-x^3-a\end{pmatrix}\mapsto
\begin{pmatrix}x^3+a+f&1\\c-2(x^3+a)f-f^2+f'\hskip0.5cm&-x^3-a-f\end{pmatrix}$$
where $f$ is any degree $2$ polynomial. Choosing $f=-a$, we  derive a unique normal form
$$\Omega=\begin{pmatrix}x^3&1\\\tilde{c}(x)&-x^3\end{pmatrix}\frac{dx}{\prod_i(x-t_i)}
\ \ \ \text{where}\ \ \ 
\frac{\tilde{c}(x)}{\prod_i(x-t_i)}=c_0+\sum_i\frac{\tau_i\theta_i^2}{x-t_i}
%=\sum_i\begin{pmatrix}\frac{t_i^3}{\tau_i}&\frac{1}{\tau_i}\\ \tau_i\theta_i^2&-\frac{t_i^3}{\tau_i}\end{pmatrix}\frac{dx}{x-t_i}
%+\begin{pmatrix}0&0\\ c_0&0\end{pmatrix}dx
\ \ \ \text{with}\ \ \ c_0\in\mathbb C$$
and constants $\tau_i$ are defined by
\begin{equation}\label{def:tau_i}
\frac{1}{\prod_i(x-t_i)}=\sum_i\frac{1}{\tau_i(x-t_i)},\ \text{i.e.}\ \tau_i=\prod_{j\not=i}(t_i-t_j).
\end{equation}
It follows that the moduli space of connections {\bf on the fixed bundle} $E=\mathcal O(-1)\oplus\mathcal O(1)$
is the affine line $\mathbb A^1\ni c_0$ (extra parameter in the above normal form). 
To get the full moduli space $\Con^{\boldsymbol{\theta}}(\mathbb P^1,\{t_1,t_2,t_3,t_4\})$ of degree $0$ bundles with the same eigenvalues (satisfying $(\theta_i^+,\theta_i^-)=(\theta_i,-\theta_i)$ and 
$\pm\theta_1\pm\theta_2\pm\theta_3\pm\theta_4\not\in\mathbb Z$), we have to patch together
this affine line $\mathbb A^1$ with the affine cubic surface $S^{\boldsymbol{\theta}}$ of Fuchsian systems given by Proposition \ref{Prop:FuchsInv}
(see remark ending example \ref{ex:NonTrivBunCon}).
Following Theorem \ref{thm:IIS}, one should obtain a smooth irreducible quasi-projective surface.
\end{example}

\subsection{Isomorphisms between moduli spaces}

There are many isomorphisms between moduli spaces of connections $\Con^{\boldsymbol{\theta}}(X,D)$
modifying spectral data. First, one can {\bf twist} a parabolic connection $(E,\nabla,\boldsymbol{l})$
by a rang $1$ logarithmic connection $(L,\zeta)$:
\begin{itemize}
\item $L$ is a line bundle of degree $k$, $L=\mathcal O(k)$,
\item $\zeta:L\to L\otimes\Omega^1(D)$ is determined by its eigenvalues
$\lambda_i$ satisfying 
$$\lambda_1+\cdots+\lambda_n+k=0.$$
\end{itemize}
The action on connection matrix (\ref{Eq:MatrixCon}) is given by
$$\Omega\mapsto\Omega+\sum_i\begin{pmatrix}\lambda_i&0\\0&\lambda_i\end{pmatrix}\frac{dx}{x-t_i}$$
and this defines an isomorphism
\begin{equation}\label{Eq:TwistIsom}
\Con^{\boldsymbol{\theta}}(X,D)\stackrel{\otimes(L,\zeta)}{\longrightarrow}\Con^{\boldsymbol{\theta}'}(X,D)\ \ \ \text{where}\ \ \ 
\left\{\begin{matrix}(\theta_i^+)'=\theta_i^++\lambda_i\\ (\theta_i^-)'=\theta_i^-+\lambda_i\end{matrix}\right.
\end{equation}
The degree is changed by $d\mapsto d'=d+2k$. In case of even degree, we can assume up to such an isomorphism
that $d=0$ and moreover $\theta_i^++\theta_i^-=0$ for all $i$ (we get a $\mathfrak{sl}_2$-connection). In the odd case,
we can assume $d=1$; then, in case $n=4$, necessarily $E=\mathcal O\oplus\mathcal O(1)$
(see remark ending example \ref{ex:NonTrivBunCon}).

The other family of isomorphisms comes from {\bf elementary transformations} of parabolic bundles.
Given a parabolic bundle $(E,\boldsymbol{l})$ over $(X,D)$, then we define 
$(E',\boldsymbol{l}'):=\mathrm{Elm}_{t_i}^-(E,\boldsymbol{l})$ by 
\begin{itemize}
\item $E'$ is the vector bundle defined by the subsheaf $E'\subset E$ generated by those sections directed by $l_i$,
i.e. defined by the exact sequence of morphisms of sheaves
$$0\to E'\to E\to E/_{l_i}\to 0$$
where $l_i$ is viewed as a sky-scrapper sheaf.
\item the new parabolic direction $l_i'\subset E'\vert_{t_i}$ is defined by the kernel of the inclusion morphism 
$E'\to E$.
\end{itemize}
A connection $\nabla$ on $E$ induces a connection $\nabla'$ on the subsheaf $E'\subset E$;
if $(E,\nabla,\boldsymbol{l})$ is parabolic, $\nabla$ preserves the direction $l_i$ and $\nabla'$
is still parabolic with respect to $(E',\boldsymbol{l}')$. Over $t_i$, the new eigenvalues are
$(\theta_i^+,\theta_i^-)':= (\theta_i^-+1,\theta_i^+)$ ($l_i'$ is now in the eigenspace corresponding to $\theta_i^-+1$).
Over other $t_j$'s, the morphism $E'\to E$ is a local bundle isomorphism
and the eigenvalues remain unchanged. This operation induces a map
\begin{equation}\label{Eq:ElmIsom}
\Con^{\boldsymbol{\theta}}(X,D)\stackrel{\mathrm{Elm}_{t_i}^-}{\longrightarrow}\Con^{\boldsymbol{\theta}'}(X,D)
\ \ \ \text{where}\ \ \ 
\left\{\begin{matrix}(\theta_i^+)'=\theta_i^-+1\\ (\theta_i^-)'=\theta_i^+\\ \text{all other}\ (\theta_j^\pm)'=\theta_j^\pm\end{matrix}\right.
\end{equation}
In local trivialization of the bundle $E$ around $t_i$, with basis $(e_1,e_2)$ such that $e_1$ generates the parabolic direction 
$l_i$ at $t_i$, we get a parabolic connection matrix of the form 
$$\Omega=\begin{pmatrix}\theta_i^+&\star\\0&\theta_i^-\end{pmatrix}\frac{dx}{x-t_i}+\text{holomorphic matrix,}
\ \ \ \text{and}\ \ \ l_i=\mathbb C\begin{pmatrix}1\\0\end{pmatrix}.$$
The new vector bundle $E'$ is generated near $t_i$ by $(e_1,(x-t_i)e_2)$ so that the connection matrix of $\nabla'$
is given by 
$$\Omega'=M^{-1}\Omega M+M^{-1}dM\ \ \ \text{where}\ \ \ 
M=\begin{pmatrix}1&0\\ 0&\hskip0.5cm  x-t_i\end{pmatrix}$$
which gives
$$\Omega'=\begin{pmatrix}\theta_i^+&0\\ \star&\hskip0.5cm  \theta_i^-+1\end{pmatrix}\frac{dx}{x-t_i}+\text{holomorphic matrix,}
\ \ \ \text{and}\ \ \ l_i=\mathbb C\begin{pmatrix}0\\1\end{pmatrix}.$$
One easily check that applying twice $\mathrm{Elm}_{t_i}^-$ gives the twist by the unique logarithmic connection 
on $\mathcal O(-1)$ having a single pole (with residue $+1$) at $t_i$:
$$\mathrm{Elm}_{t_i}^-\circ \mathrm{Elm}_{t_i}^-(E,\nabla,\boldsymbol{l})=(\mathcal O(-1),\frac{dx}{x-t_i})\otimes(E,\nabla,\boldsymbol{l}).$$
Note that $\mathrm{Elm}_{t_i}^-$ decreases the degree of the vector bundle by $-1$ (whence the sign);
one can also define $\mathrm{Elm}_{t_i}^+(E,\boldsymbol{l}):=(\mathcal O(-1),\frac{dx}{x-t_i})\otimes\mathrm{Elm}_{t_i}^-(E,\boldsymbol{l})$ increasing the degree of the vector bundle by $+1$, it is actually the inverse map
$$\mathrm{Elm}_{t_i}^+\circ \mathrm{Elm}_{t_i}^-(E,\nabla,\boldsymbol{l})\simeq(E,\nabla,\boldsymbol{l}).$$
In fact, all these isomorphisms commute together.

One can also define an isomorphism permuting $\theta_i^+$ and $\theta_i^-$, modifying the parabolic structure
(we switch to the other eigenspace) and get a non abelian group. We omit this from our discussion.
Finally, up to isomorphism, we can always assume degree $d=0$ and $\theta_i^++\theta_i^-=0$ for all $i$.
However, this may be not the best choice as we shall see.

\section{The Painlev\'e case $n=4$}

In order to get an explicit irreducible moduli space, it is obviously better to choose $d=1$
instead of $d=0$ since then, we can work on the single vector bundle $E=\mathcal O\oplus\mathcal O(1)$
(see remark ending example \ref{ex:NonTrivBunCon}).

\subsection{An explicit construction}\label{subsec:Explicit4}

We start proceeding like in example \ref{ex:ModuliConF2}. 
The connection matrix writes $\Omega=A(x)\frac{dx}{\prod_i(x-t_i)}$ where
$$A(x)=\begin{pmatrix}a(x)&b(x)\\c(x)&\hskip0.5cm-x^3+d(x)\end{pmatrix} \ \ \ \text{with}\ \ \ 
\left\{\begin{matrix}b\ \text{polynomial of degree}\ \le1\hfill\\ a,d\ \text{polynomials of degree}\ \le2\\ 
c\ \text{polynomial of degree}\ \le3\end{matrix}\right.$$ 
Again, $b(x)\not\equiv0$ otherwise the connection is reducible. Assume for the moment that $b(x)$ is not constant
so that we can normalize it to $b(x)=x-q$. The action by unipotent matrices is not enough to kill $a(x)$ however
(the freedom $f(x)$ like in example \ref{ex:ModuliConF2} is now of degree $1$).
The idea is to apply an elementary transformation $\mathrm{Elm}_{q}^+$ at the parabolic direction given by $\mathcal O(1)$.
Concretely, we apply the bundle birational isomorphism given by 
$$M=\begin{pmatrix}1&0\\0&\frac{1}{x-q}\end{pmatrix}$$
and obtain the new connection matrix
$$M^{-1}\Omega M+M^{-1}dM=\begin{pmatrix}\frac{a(x)}{\prod_i(x-t_i)}&\frac{1}{\prod_i(x-t_i)}\\ \frac{(x-q)c(x)}{\prod_i(x-t_i)}&\hskip0.5cm\frac{-x^3+d(x)}{\prod_i(x-t_i)}-\frac{1}{x-q}\end{pmatrix}dx$$
On the new bundle $E'=\mathcal O\oplus\mathcal O(2)$, we have more automorphisms
and we are able to kill the polynomial coefficient $a(x)$ by means of unipotent automorphisms. We now get the normal form
\begin{equation}\label{Eq:ConNormForm4}
\Omega'=\begin{pmatrix}0&\frac{1}{\prod_i(x-t_i)}\\ \frac{\tilde{c}(x)}{(x-q)\prod_i(x-t_i)}&\hskip0.5cm\frac{-x^3+\tilde{d}(x)}{\prod_i(x-t_i)}-\frac{1}{x-q}\end{pmatrix}dx
\end{equation}
with $\tilde{c}(x)$ polynomial of degree $5$. Taking into account the spectral data, we get
$$\Omega=\sum_i\begin{pmatrix}0&\frac{1}{\tau_i}\\ -\tau_i\theta_i^+\theta_i^-&\hskip0.5cm\theta_i^++\theta_i^-\end{pmatrix}\frac{dx}{x-t_i}
+\begin{pmatrix}0&0\\ p&-1\end{pmatrix}\frac{dx}{x-q}+\begin{pmatrix}0&0\\ c_0&0\end{pmatrix}dx$$
for constants $p,c_0\in\mathbb C$. Here, we have also assumed $q\in\mathbb C\setminus\{t_1,t_2,t_3,t_4\}$.
Indeed, when $q=t_i$, we apply an elementary transformation over the pole which might change the eigenvalues.
Finally, note that $x=q$ must be an apparent singular point, which means, by Fuchs' local theory, that around $x=q$
the kernel of the residual part is also an eigenvector for the constant part of the connection matrix: 
$$\Omega\cdot\begin{pmatrix}1\\ p\end{pmatrix}=\star%\frac{p}{\prod_i(q-t_i)}
\begin{pmatrix}1\\ p\end{pmatrix}+o(x-q);$$
this fixes the value of $c_0$. Finally, the initial connection is determined by the parameters $(p,q)$ which 
have the following geometrical meaning:
\begin{itemize}
\item $q$ stands for the position of the apparent singular point of the new connection,
\item $p$ stands for the kernel of the residue of $\nabla$ over $x=q$.
\end{itemize}
Note also that, over $x=t_i$, the two eigendirections of $\nabla$ are given by 
$$\begin{pmatrix}1\\ \tau_i\theta_i^+\end{pmatrix}\ \ \ \text{and}\ \ \ \begin{pmatrix}1\\ \tau_i\theta_i^-\end{pmatrix}$$
corresponding respectively to eigenvalues $\theta_i^+$ and $\theta_i^-$.

Let us identify $\mathcal O(2)=\Omega^1(D)$, and consider the projectivization $\mathbb P(\mathcal O\oplus\Omega^1(D))$
of our vector bundle $E$ as the fiber-compactification of $\Omega^1(D)$: the line bundles in $E$ generated by 
$$\begin{pmatrix}1\\ 0\end{pmatrix},\ \begin{pmatrix}1\\ 1\end{pmatrix}\ \ \ \text{and}\ \ \ \begin{pmatrix}0\\ 1\end{pmatrix}$$
respectively correspond to the zero section, the section $\frac{dx}{\prod_i(x-t_i)}$ and the section at infinity.
If $S$ denotes the associated ruled surface
(total space of the projective bundle), then the total space of $\Omega^1(D)$ identifies with $S\setminus\Sigma$
where $\Sigma$ is the section at infinity. On fibers $F_i:\{x=t_i\}$, a natural affine coordinate if given by the residue:
the section $\frac{dx}{\prod_i(x-t_i)}$ has residue $\frac{1}{\tau_i}$ at $x=t_i$. Therefore, the two eigendirections
over $x=t_i$ correspond to the points $s_i^+,s_i^-\in F_i$ satisfying $x=t_i$ and 
$\mathrm{Res}_{t_i}(s_i^{\pm})=\theta_i^{\pm}$. Over $x=q$, the kernel of the residual matrix also
defines a point $s$:
$$s_i^{\pm}\ :\ \begin{pmatrix}1\\ \tau_i\theta_i^{\pm}\end{pmatrix}\ \ \ \text{and}\ \ \ s\ :\ \begin{pmatrix}1\\ p\prod_i(q-t_i)\end{pmatrix}.$$
We have just defined a birational map
$$\Con^{\boldsymbol{\theta}}(X,D)\dashrightarrow S\ ; \ (E,\nabla,\boldsymbol{l})\mapsto s.$$
Now consider the blow-up of the surface $\pi:\hat{S}\to S$ at all $8$ points $s_i^\pm$:
we denote by $E_i^\pm$ the exceptional divisors and $\hat{\Sigma}$, $\hat{F}_i$ 
the strict transforms.
When $\theta_i^+=\theta_i^-$, then first blow-up the point $s_i:=s_i^+=s_i^-$, and then 
blow-up the intersection point between the exceptional divisor $E_i$
and the strict transform of the fiber $F_i$; denote by $E_i'$ the last exceptional divisor.
Then the birational map above induces an isomorphism: 
$$\Con^{\boldsymbol{\theta}}(X,D)\stackrel{\sim}{\longrightarrow} M^{\boldsymbol{\theta}}:=\hat{S}\setminus\hat{\Sigma}\cup \hat{F}_1\cup \hat{F}_2\cup \hat{F}_3\cup \hat F_4.$$
For a general point $s\in M^{\boldsymbol{\theta}}$, formula (\ref{Eq:ConNormForm4})
gives a parabolic connection with an extra apparent singular point; after elementary transformation over $x=q$,
we get a parabolic connection in $\Con^{\boldsymbol{\theta}}(X,D)$. When $q\to \infty$, there is a limit point in 
$\Con^{\boldsymbol{\theta}}(X,D)$ provided that $s$ tends to an affine point of the fiber. When $q\to t_i$,
one can find a limit point in $\Con^{\boldsymbol{\theta}}(X,D)$ if, and only if, $s$ tends to one of the exceptional divisors,
not to $\hat F_i$; over the point $s_i^-$, we get all parabolic connections in $\Con^{\boldsymbol{\theta}}(X,D)$ such that
the parabolic $l_i$ lies in the destabilizing subbundle $\mathcal O(1)$.

With our notations, the $2$-form $dp\wedge dq$ has polar divisor $2\hat{\Sigma}+\hat F_1+ \hat F_2+ \hat F_3+ \hat F_4$
and is non zero at any other point: it defines a  non degenerate holomorphic volume form on the moduli space
that should better viewed as a symplectic structure.

\subsection{Description and properties of moduli spaces}

Fix $(X,D)=(\mathbb P^1,\{t_1,t_2,t_3,t_4\})$ as before and consider the Hirzebruch surface $S$
given by the total space of the projective bundle $\mathbb P(\mathcal O\oplus\Omega^1(D))$.
This surface may be obtained from the total space of the projective bundle $\mathbb P(\mathcal O\oplus\Omega^1)$
by applying $4$ elementary transformations directed by the section at infinity $\mathcal O$; from this point of view,
the Liouville form on the total space of $\Omega^1$ induces a rational $2$-form $\omega$ on $\mathbb P(\mathcal O\oplus\Omega^1(D))$ with polar divisor $2\Sigma+F_1+ F_2+ F_3+ F_4$: as before, $\Sigma$ is the section defined
by $\mathcal O$ and $F_i$ is the fiber over $t_i$. An affine chart is defined by 
$$(x,y):\mathbb P(\mathcal O\oplus\Omega^1)\setminus \Sigma\cup F_\infty\to\mathbb C\times\mathbb C$$
where $F_\infty$ is the fiber over $x=\infty$ and $y$ is normalized so that 
\begin{itemize}
\item $y=0$ corresponds to the section defined by $0\in H^0(X,\mathcal O\oplus\Omega^1(D))$,
\item $y=1$ corresponds to the section defined by $\frac{dx}{\prod_i(x-t_i)}\in H^0(X,\mathcal O\oplus\Omega^1(D))$
\end{itemize} 
(and $y=\infty$ is defined by $\mathcal O$). For each finite point $y\in F_i$, we can associate the residue
$\mathrm{Res}_{t_i}(y\frac{dx}{\prod_i(x-t_i)})=\frac{y}{\tau_i}$ (see definition (\ref{def:tau_i})). This gives 
us a natural parametrization of fibers $F_i$.

Choose $2$ points $s_i^{\pm}$ on each fiber $F_i$ (possibly, $s_i^+=s_i^-$). 
Then, blow-up these $8$ points and denote by $\hat S$ the blow-up surface, and still denote by $\hat{\Sigma}$
and $\hat F_i$ the strict transforms. Then consider the open part 
$$M^{\boldsymbol{\theta}}:=\hat{S}\setminus\hat{\Sigma}\cup \hat F_1\cup \hat F_2\cup \hat F_3\cup \hat F_4.$$
The $2$-form $\omega$ extends as a holomorphic symplectic form on $M$ with polar divisor
$$(\omega)_\infty=2\hat{\Sigma}+ \hat F_1+ \hat F_2+ \hat F_3+ \hat F_4.$$
Consider on each fiber $F_i$ the point $m_i$ defined by the arithmetic mean of $s_i^+$ and $s_i^-$.
Precisely, if we denote $\mathrm{Res}_{t_i}(s_i^{\pm})=\theta_i^{\pm}$, then $m_i$ is defined by
$\mathrm{Res}_{t_i}(m_i)=\frac{\theta_i^++\theta_i^-}{2}$.

Since $S\setminus\Sigma$ is the total space of $\Omega^1(D)$, there is a one-to-one correspondence
between global holomorphic sections of $\Omega^1(D)$ and global sections of $S\to\mathbb P^1$
not intersecting $\Sigma$. These sections form a $3$-dimensional family.
Assume first that there exists such a section passing through all $4$ points $m_i$;
this is equivalent to say $\sum_i(\theta_i^++\theta_i^-)=0$.
In this case, it can be checked that $2\hat{\Sigma}+ \hat F_1+ \hat F_2+ \hat F_3+ \hat F_4$ is the degenerate
fiber of an elliptic fibration on $\hat S$. In this case, the open part $M$ may be viewed as a moduli
space of Higgs bundles and this fibration is the well-known Hitchin fibration.
The fibration is defined by those elements of the linear system $\vert 2\hat{\Sigma}+ \hat F_1+ \hat F_2+ \hat F_3+ \hat F_4\vert$
on $S$ passing through the $8$ points $s_i^{\pm}$. 

Assume now that $\sum_i(\theta_i^++\theta_i^-)\not=0$. Then, after applying a bundle isomorphism of the form
$y\mapsto cy$, we may assume $\sum_i(\theta_i^++\theta_i^-)=1$. If $\theta_1^{\pm}+\theta_2^{\pm}+\theta_3^{\pm}+\theta_4^{\pm}\not\in\mathbb Z$ whatever the choice of signs, the surface $M^{\boldsymbol{\theta}}$ identifies with the moduli space of connections $\Con^{\boldsymbol{\theta}}(X,D)$ discussed in the previous section. On the other hand, when $\theta_1^{\pm}+\theta_2^{\pm}+\theta_3^{\pm}+\theta_4^{\pm}\in\mathbb Z$ for some (maybe several) choices of signs, 
then $M^{\boldsymbol{\theta}}$ can still be viewed
as a moduli space of stable connections (we have to delete some reducible connection to obtain a GIT moduli space, see \cite{IIS}).
By the Riemann-Hilbert correspondance (see \cite{IIS}), there is an analytic mapping
$$RH:M^{\boldsymbol{\theta}}\longrightarrow\chi\subset\mathbb C^3$$
where $\chi$ is an affine cubic surface. For generic $\boldsymbol{\theta}$, this is an analytic isomorphism;
for special values of $\boldsymbol{\theta}$ like above, 
the affine surface $\chi$ is singular and the map $RH$ is a minimal analytic resolution of singularities.
It follows that the only occurence of complete curves in $M^{\boldsymbol{\theta}}$ is for special values of $\boldsymbol{\theta}$:
they come from exceptional divisors in the resolution of singularities of $\chi$: they are rational, and there are
at most $4$ such curves in a given $M^{\boldsymbol{\theta}}$. In particular, the divisor 
$2\hat{\Sigma}+ \hat F_1+ \hat F_2+ \hat F_3+ \hat F_4$ is no more contained in an elliptic fibration.

The surface $M^{\boldsymbol{\theta}}$ is not affine however. In fact, 
there are no non-constant regular functions on $M^{\boldsymbol{\theta}}$ (see \cite{ArinkinLysenko}).
This mainly follows from numerical properties
$$(2\hat{\Sigma}+ \hat F_1+ \hat F_2+ \hat F_3+ \hat F_4)\cdot \hat{\Sigma}=(2\hat{\Sigma}+ \hat F_1+ \hat F_2+ \hat F_3+ \hat F_4)\cdot \hat F_i=0;$$
indeed, a rational function on $\hat S$ whose polar divisor is contained in the support of $2\hat{\Sigma}+ \hat F_1+ \hat F_2+ \hat F_3+ \hat F_4$
would define a fibration of complete curves inside the complement $M^{\boldsymbol{\theta}}$, contradicting
the (almost) non existence of complete curves. This generalize Serre example. 

\subsection{How complete curves arise for special parameters}

When $\theta_i^+=\theta_i^-$ for some $i$, then the two points $s_i^+$ and $s_i^-$ coincide in $S$.
In this case, we have to blow-up twice over $s_i^+=s_i^-$ to get the surface $\hat S$. Precisely,
we first blow-up this point to get an exceptional divisor $E_i$, and then blow-up the intersection
point between $E_i$ and the strict transform of $F_i$; we get a second exceptional divisor $E_i'$
and still denote by $E_i$ and $\hat F_i$ the strict transforms in the resulting surface $\hat S$.
We define similarly the open set $M^{\boldsymbol{\theta}}:=\hat S\setminus \hat{\Sigma}\cup \hat F_1\cup \hat F_2\cup \hat F_3\cup \hat F_4$. 
Then, $E_i$ is a rational curve in $M^{\boldsymbol{\theta}}$ having $-2$ self-intersection number.
This curve corresponds to the locus of those parabolic connections  $(E,\nabla,\boldsymbol{l})$
for which the residual matrix at $t_i$ is the scalar matrix $\theta_i^+\cdot I$.

Similarly, when $\theta_i^+=\theta_i^-+1$, from Fuchs' local theory, we can check that a generic element 
$(E,\nabla,\boldsymbol{l})$ of $\Con^{\boldsymbol{\theta}}(X,D)$ will have a logarithmic singular point 
at $t_i$ (the local monodromy has a non trivial Jordan block); however, for some special parabolic connections,
the singular point is not logarithmic anymore (it has diagonal monodromy). The locus of those 
special connections is a complete and smooth rational curve in $M^{\boldsymbol{\theta}}$,
namely the strict transform of the curve $C\subset S$ characterized as follows:
\begin{itemize}
\item $q:C\to \mathbb P^1$ has degree $2$ and $C$ does not intersect $\Sigma$;
\item $C$ intersects $F_i$ at $s_i^+$ and has two smooth local branches at this point;
\item $C$ intersects other fibers $F_j$ at both $s_j^+$ and $s_j^-$.
\end{itemize}
The first condition fixes the linear system on $S$ that contains the curve: $C\in\vert2\Sigma+ F_1+ F_2+ F_3+ F_4\vert$. 
The two other conditions characterize $C$ in this linear system. 
After blowing-up, the resulting curve $\hat C$ becomes smooth and does not intersect
strict transforms $F_j$'s anymore.

When $\theta_i^+-\theta_i^-\in\mathbb Z$, the story is the same. Let us just characterize
the special curve $C$ in the case $\theta_i^+=\theta_i^-+2$:
\begin{itemize}
\item $q:C\to \mathbb P^1$ has degree $4$ and $C$ does not intersect $\Sigma$;
\item $C$ intersects $F_i$ three times at $s_i^+$ and one time at $s_i^-$;
\item $C$ intersects other fibers $F_j$ twice at both $s_j^+$ and $s_j^-$.
\end{itemize}
This can be checked by straightforward formal computation.

Another occurence of complete curves in $M^{\boldsymbol{\theta}}$ comes from 
reducible connections. When say $\theta_1^++\theta_2^++\theta_3^++\theta_4^+=0$,
then some reducible connections arise in $\Con^{\boldsymbol{\theta}}(X,D)$:
each of them stabilize a trivial subbundle $\mathcal O\subset E$ and induces on it
a logarithmic rank $1$ connection with eigenvalue $\theta^+_i$ at $t_i$.
The locus of these reducible connections form a 
section $\hat C$ for the projection $q:M^{\boldsymbol{\theta}}\to \mathbb P^1$.
The corresponding curve $C\subset S$ is the section of $q:C\to \mathbb P^1$ that does not intersect $\Sigma$
and intersecting each $F_i$ at $s_i^+$.

\subsection{Link with Fuchsian systems}

As we shall see here, if we delete one of the exceptional divisors $E_i^\pm\to s_i^\pm$ from $M^{\boldsymbol{\theta}}$,
then the surface becomes affine.

Consider the moduli space $\Con^{\boldsymbol{\theta}}(X,D)$ with eigenvalues 
$\boldsymbol{\theta}$ satisfying
$$\left\{\begin{matrix}\theta_i^+=\theta_i\\ \theta_i^-=-\theta_i\end{matrix}\right.\ \ \ \text{except}
\ \ \ \left\{\begin{matrix}\theta_4^+=-\theta_4\\ \theta_4^-=\theta_4-1\end{matrix}\right.$$
After applying an elementary transformation 
$$\mathrm{Elm}_{t_4}^-\ :\ \Con^{\boldsymbol{\theta}}(X,D)\stackrel{\sim}{\longrightarrow} \Con^{\boldsymbol{\theta}'}(X,D)$$
we obtain new eigenvalues 
$$\left\{\begin{matrix}(\theta_i^+)'=\theta_i\\ (\theta_i^-)=-\theta_i\end{matrix}\right.\ \ \ \text{for all }\ i=1,2,3,4.$$
This latter moduli space splits into the disjoint union of those connection defined on the trivial bundle, namely 
Fuchsian systems discussed in section \ref{sec:FuchsianSyst},
and those defined on $E=\mathcal O(-1)\oplus\mathcal O(1)$ (see section \ref{ex:ModuliConF2}).
In fact, a parabolic connection $(E=\mathcal O\oplus\mathcal O(1),\nabla,\boldsymbol{l})$ of $\Con^{\boldsymbol{\theta}}(X,D)$
is sent to a connection on the non trivial bundle by the elementary transformation
if, and only if, its parabolic $l_4$ lie on the subbundle $\mathcal O(1)\subset E$.
The locus of those special parabolic connections in $M^{\boldsymbol{\theta}}$
is given by the exceptional divisor $E_4^-\to s_4^-$ (or we better should say the affine
part of this divisor once we have deleted the intersection point with $F_4$).
Therefore, $E_4^-$ identifies via $\mathrm{Elm}_{t_4}^-$ to the moduli space 
$\mathbb A^1$ of those connections on $\mathcal O(-1)\oplus\mathcal O(1)$ discussed in section \ref{ex:ModuliConF2}.
Moreover, after deleting $E_4^-$ from $M^{\boldsymbol{\theta}}$, we get an isomorphism with the moduli space of 
Fuchsian systems computed in section \ref{sec:FuchsianSyst}, i.e. a cubic affine surface:
$$M^{\boldsymbol{\theta}}\setminus E_4^-\stackrel{\mathrm{Elm}_{t_4}^-}{\stackrel{\sim}{\longrightarrow}}\Fuchs^{\boldsymbol{\theta}'}(X,D)\stackrel{\sim}{\longrightarrow}S^{\boldsymbol{\theta}}\subset\mathbb C^4.$$
We can do in a similar way with other elementary transformations $\mathrm{Elm}_{t_i}^\pm$
(that we have to compose with a convenient twist in order to fit with $\mathfrak{sl}_2$-systems 
discussed in  section \ref{sec:FuchsianSyst}) and we promtly deduce that $M^{\boldsymbol{\theta}}$
becomes a cubic affine surface once we delete any one of the exceptional divisors $E_i^\pm$.

\section{The Garnier case $n>4$}

The computation of section \ref{subsec:Explicit4} generalizes as follows.
Fix $D=t_1+\cdots+t_n$ a reduced divisor one $X:=\mathbb P^1$, $n\ge4$.
Fix eigenvalues $\boldsymbol{\theta}$,  $\theta_i^\pm\in\mathbb C$ for each $i=1,\ldots,n$,
with $\sum_{i=1}^n(\theta_i^++\theta_i^-)+1=0$. Assume $\boldsymbol{\theta}$
generic for simplicity: 
$$\theta_i^+-\theta_i^-\not\in\mathbb Z\ \ \ \text{and}\ \ \ \theta_1^{\pm}+\cdots+\theta_n^{\pm}\not\in\mathbb Z.$$
Consider the moduli space $\Con^{\boldsymbol{\theta}}(X,D)$ of corresponding connections.
Note that we can omit parabolic structure from the discussion: under genericity of $\boldsymbol{\theta}$,
the parabolic structure $\boldsymbol{l}$ is determined by the connection $(E,\nabla)$.

We would like to explicitely describe the generic connection and derive an explicit open subset 
of $\Con^{\boldsymbol{\theta}}(X,D)$. Let us first assume $E=\mathcal O\oplus\mathcal O(1)$.
The special subbundle $\mathcal O(1)\subset E$ is stabilized by $\nabla$ at exactly $n-3$
points counted with multiplicity; these are the zeroes of the composition map
$$L=\mathcal O(1)\stackrel{\text{inclusion}}{\longrightarrow} E\stackrel{\nabla}{\longrightarrow} E\otimes\Omega^1(D)
\stackrel{\text{quotient}}{\longrightarrow} (E/L)\otimes\Omega^1(D).$$
These points are also the extra apparent singular points for the fuchsian scalar equation derived
from $\nabla$ and the cyclic vector $L$.
Let us consider the case where these zeroes and the poles $t_i$'s are two-by-two distinct;
denote the by $q_1,\ldots,q_{n-3}$ the apparent points. We also assume $q_j\not=\infty$ for simplicity.

Consider now the parabolic structure $\tilde{\boldsymbol{l}}$ defined over each $q_j$ by the special line bundle $L$.
Still assuming all $t_i$'s and $q_j$'s two-by-two distinct, apply a positive elementary transformation
at each $q_j$ directed by $\tilde{\boldsymbol{l}}$ (i.e. by $L=\mathcal O(1)\subset E$).
Let $(E',\nabla',\tilde{\boldsymbol{l}}')$ be the resulting parabolic connection; note that $E'=\mathcal O\oplus\mathcal O(n-2)$.

\begin{prop}The connection $(E',\nabla')$ is given, up to bundle automorphism, by the connection matrix
(we follow notations of section \ref{subsec:ConMatrix})
$$\Omega=\sum_{i=1}^n\begin{pmatrix}0&\frac{1}{\tau_i}\\ -\tau_i\theta_i^+\theta_i^-&\hskip0.5cm\theta_i^++\theta_i^-\end{pmatrix}\frac{dx}{x-t_i}
+\sum_{j=1}^{n-3}\begin{pmatrix}0&0\\ p_j&-1\end{pmatrix}\frac{dx}{x-q_j}+\begin{pmatrix}0&0\\ c(x)&0\end{pmatrix}dx$$
where $c(x)$ is polynomial of degree $n-4$, determined by the fact that all poles $q_j$'s are non
logarithmic (with scalar monodromy). The connection is determined by the new parabolic structure $\tilde{\boldsymbol{l}}'$ defined
by kernel of residual matrices over the $q_j$'s.
\end{prop}

\begin{proof}We proceed like in section \ref{subsec:Explicit4}.
On the initial bundle $E=\mathcal O\oplus\mathcal O(1)$, the connection matrix writes $\Omega=A(x)\frac{dx}{\prod_i(x-t_i)}$ where
$$A(x)=\begin{pmatrix}a(x)&b(x)\\c(x)&\hskip0.5cm-x^{n-1}+d(x)\end{pmatrix} \ \ \ \text{with}\ \ \ 
\left\{\begin{matrix}\deg(b)\ \le n-3  \hfill\\ \deg(a),\deg(d)\ \le n-2\\ 
\deg(c)\ \le n-1\end{matrix}\right.$$ 
Then, up to a diagonal automorphism, we have $b(x)=\prod_{j=1}^{n-3}(x-q_j)$.
We now apply the $n-3$ elementary transformations and get the new connection matrix
$$\begin{pmatrix}\frac{a(x)}{\prod_i(x-t_i)}&\frac{1}{\prod_i(x-t_i)}\\ \frac{c(x)\prod_j(x-q_j)}{\prod_i(x-t_i)}&\hskip0.5cm-\frac{-x^{n-1}+d(x)}{\prod_i(x-t_i)}-\sum_{j=1}^{n-3}\frac{1}{x-q_j}\end{pmatrix}dx$$
On the new bundle $E'=\mathcal O\oplus\mathcal O(n-2)$, we can use unipotent automorphisms of $E'$
to kill the polynomial coefficient $a(x)$. We now get the unique normal form
\begin{equation}\label{Eq:ConNormFormN}
\Omega'=\begin{pmatrix}0&\frac{1}{\prod_i(x-t_i)}\\ \frac{\tilde{c}(x)}{\prod_i(x-t_i)\prod_j(x-q_j)}&\hskip0.5cm\frac{-x^{n-1}+\tilde{d}(x)}{\prod_i(x-t_i)}-\sum_{j=1}^{n-3}\frac{1}{x-q_j}\end{pmatrix}dx
\end{equation}
with $\tilde{c}(x)$ and $\tilde{d}(x)$ polynomials of degree $3n-7$ and $n-2$ respectively. Now, taking into account the spectral data, we get the normal form of the statement.
\end{proof}

From this statement, one easily derive the following open subset of $\Con^{\boldsymbol{\theta}}(X,D)$.
Consider the surface $S$ defined by the total space of $\mathbb P(\mathcal O\oplus\Omega^1(D))$
and denote by $\Sigma$ the section defined by $\Omega^1(D)$: $S\setminus\Sigma$ naturally identifies 
with the total space of $\Omega^1(D)$. In particular, the affine part of the fiber $F_i$ over $t_i$ has the natural
chart $\mathrm{Res}_{t_i}:F_i\setminus\Sigma\stackrel{\sim}{\longrightarrow}\mathbb C$ given by 
taking the residue of sections of $\Omega^1(D)$; we define two points $s_i^{\pm}\in F_i$ 
by  $\mathrm{Res}_{t_i}(s_i^\pm)=\theta_i^\pm$. Let $\pi=\hat S\to S$ be the blow-up
of $S$ at all $2n$ points and denote by $\hat{\Sigma}$ and $\hat F_i$ the strict transforms.
Then consider the open part 
$$M^{\boldsymbol{\theta}}:=\hat S\setminus(\hat{\Sigma}\cup \hat F_1\cup\cdots\cup\hat F_n)$$
and the symetric product 
$$(M^{\boldsymbol{\theta}})^{(n-3)}:=M^{\boldsymbol{\theta}}\times\cdots\times M^{\boldsymbol{\theta}}/\mathrm{Perm}(1,\cdots,n).$$
Then the above normal form defines an embedding 
$$(M^{\boldsymbol{\theta}})^{(n-3)}\setminus\Delta\hookrightarrow \Con^{\boldsymbol{\theta}}(X,D)$$
where $\Delta$ is the ``codimension one diagonal'' defined by those $\{q_i,p_i\}$ for which $q_i=q_j$
for some $i\not=j$. This is the point of view developped by Oblezin in \cite{Oblezin} where the compactification
of this picture is discussed. Another point of view has been recently studied by Saito and the author in 
\cite{LoraySaito}.

%%%%%%%%%%%%%%%%%%%%%%%%%%%%%%%%%%%%%%%%%%%%%%%%%%

\end{document}